\documentclass[12pt]{amsart}
\usepackage{amsmath}
\usepackage{amssymb} 
\usepackage{epsbox}
\usepackage{amsthm}
\usepackage{graphicx}

 \newtheorem{thm}{Theorem}
 \newtheorem{cor}{Corollary}
 \newtheorem{defn}{Definition}
 \newtheorem{prop}{Propotition}
 \newtheorem{lem}{Lemma}

 {\theoremstyle{definition}
 \newtheorem{rem}{Remark}}
 {\theoremstyle{definition}
 \newtheorem{ack}{Acknowledgement}}
 
\begin{document}  

\title{Braid ordering and the geometry of closed braid}
\author{Tetsuya Ito}
\address{Graduate School of Mathematical Science, University of Tokyo, 3-8-1 Komaba, Meguro-ku, Tokyo, 153-8914, Japan}
\email{tetitoh@ms.u-tokyo.ac.jp}
\subjclass[2000]{Primary~57M25, Secondary~57M50}
\keywords{Braid groups; Dehornor ordering; essential surface; Braid foliation; Nielsen-Thurston classification}
 
\begin{abstract}
   The relationships between braid ordering and the geometry of its closure is studied. We prove that if an essential closed surface $F$ in the complements of closed braid has relatively small genus with respect to the Dehornoy floor of the braid, $F$ is circular-foliated in a sense of Birman-Menasco's Braid foliation theory. As an application of the result, we prove that if Dehornoy floor of braids are larger than three, Nielsen-Thurston classification of braids and the geometry of their closure's complements are in one-to-one correspondence. Using this result, we construct infinitely many hyperbolic knots explicitly from pseudo-Anosov element of mapping class groups.

\end{abstract}
\maketitle

\section{Introduction}
   In this paper we study an application of left-invariant total ordering of braid group $B_{n}$ to knot theory. An left-invariant total ordering, called Dehornoy ordering is an interesting structure of braid groups found in 1990's by Dehornoy (\cite{d}). Today the ordering is interpreted in both algebraic and geometric way and these interpretation shows the Dehornoy ordering is a quite natural structure of braid groups.

 Braid groups are very important tool for studying links. Many link invariants such as Jones polynomial are first constructed via closed braid representative of links, so it is natural to study an application of this interesting structure of braid groups to knot theory. In \cite{mn},\cite{ma} it is found that braid ordering can be seen as some kinds of restriction for admissibility of braid moves. Braid moves are moves of closed braids which do not change their representing link types. Most typical (and important) example of braid moves are {\it exchange move} or {\it destabilization, stabilization}, which is shown in figure \ref{braidmove}. These works suggest that there would be some relationships between Birman-Menasco's braid foliation theory and ordering of braid groups because braid moves appears naturally in braid foliation theory (\cite{bf}). For example, exchange move is used to make braid foliation simple and it is one of the most fundamental operation in Birman-Menasco's braid foliation theory. Indeed, we investigate relationship between braid foliation and left invariant total ordering of braid groups, and obtain following result:

\begin{figure}[htbp]
 \begin{center}
  \includegraphics[width=120mm]{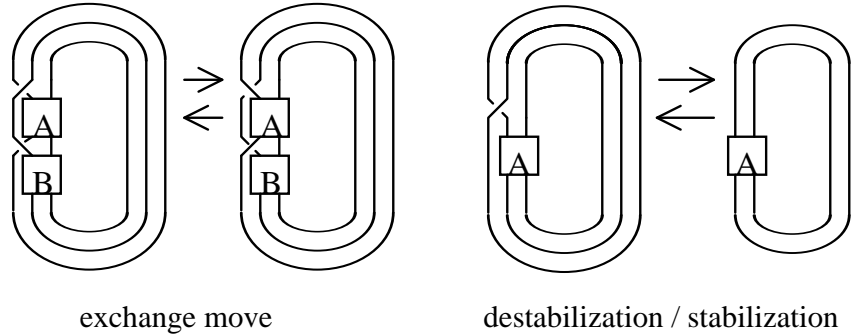}
 
 \end{center}
 \caption{exchange move, stabilization and destabilization}
 \label{braidmove}
\end{figure}

\begin{thm}
\label{thm:main1}
   Let $\beta \in B_{n}$ be a braid and $F$ be an essential, genus $g$ closed surface in the complement of closed braid $\widehat{\beta}$.
\begin{enumerate}
\item If $F$ is tiled, then $[\beta]_{D} < g+2$.
\item If $F$ is mixed-foliated, $[\beta]_{D} < 2g+1$.
\item If $[\beta]_{D} \geq 2g+1$, then $F$ is circular foliated.
\end{enumerate} 
\end{thm}

  In above statements, $[\;]_{D}$ represents {\it Dehornoy floor} of $\beta$, which is a non-negative integer defined in section 2. This integer is determined by the Dehornoy ordering of braid groups, and efficiently computable. The notion of tiled, mixed foliated or circular foliated are derived from Birman-Menasco's braid foliation theory, which will be explained in section 3. Proof of Theorem \ref{thm:main1} will be given in section 4.

    Theorem \ref{thm:main1} means braid ordering gives some information about the geometry of closed braid complements, and especially if Dehornoy floor is relatively large with respect to its genus, then the surface can be positioned to special position, which is simple .

   We remark that Dehornoy floor might vary under even conjugation, stabilization or destabilization, so it is far from knot invariant. Indeed, Dehornoy floor of braids obtained by performing stabilization always vanish. Nevertheless theorem \ref{thm:main1} seems to suggests that braid ordering provide some important information about its closure which can not be obtained from known invariants of links.
   
   As an application of theorem \ref{thm:main1} we obtain a sufficient condition for Nielsen-Thurston classification determines the geometry of its closure's complement. Nielsen-Thurston theorem states braids are classified into three types by their dynamics; {\it periodic, reducible} and {\it pseudo-Anosov}. On the other hand, a famous Thurston's hyperbolization theorem shows knots are classified into following three types by the geometry of their complements; {\it Torus knots, Satellite knots,} and {\it Hyperbolic knots}. Although there are many similarities between these two classifications, in general there are no one-to-one correspondence of them. In figure \ref{fig:example}, we give an example of periodic, reducible and pseudo-Anosov braids whose closure represents $(2,3)$-Torus knot.
 Therefore the following question naturally arises:

{\it    When the Nielsen-Thurston classification of braids completely determine the geometrical types of its closure? }

   Theorem \ref{thm:main2} partially answers the question.
\begin{figure}[htbp]
 \begin{center}
  \includegraphics[width=60mm]{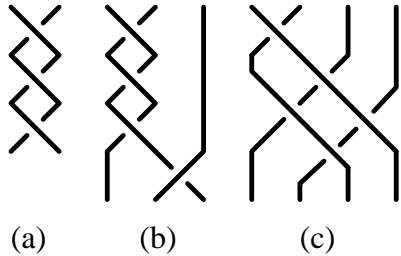}
 \end{center}
 \caption{Examples of periodic (a), pseudo-Anosov (b), reducible (c) braids all of whose closure are $(2,3)$-Torus knots}
 \label{fig:example}
\end{figure}

\begin{thm}
\label{thm:main2}
Let $\beta \in B_{n}$ be a braid whose closure is a knot.
 If $[\beta]_{D} \geq 3$, following holds:
 \begin{enumerate}
  \item $\beta$ is periodic if and only if $\widehat{\beta}$ is a torus knot.
  \item $\beta$ is reducible if and only if  $\widehat{\beta}$ is a satellite knot.
  \item $\beta$ is pseudo-Anosov if and only if $\widehat{\beta}$ is a hyperbolic knot.
 \end{enumerate}
 
 Especially, under this assumption, Satellite knots are all braided-satellite.
\end{thm} 

   We recall that satellite knot $K$ is called braided satellite if $K$ admit closed braid representatives whose axis does not intersect some essential tori in the complement of $K$.  
   This theorem implies, under the condition $[\beta]_{D} \geq 3$, the Nielsen-Thurston classification of braid groups and the geometric classification of their closures are in one-to-one correspondence.  
   In \cite{l}, Los studied the relationship between these two classifications in the case knots are represented by closed braid with minimum braid index. Main result of \cite{l} shows under the condition of braid index, two classifications are related each other by performing exchange move. However, in general to determine braid index of given knots or links are very difficult question and to see whether or not given braids admit exchange move is also difficult, so the relationships between two classifications cannot be seen directly. In this regard, theorem \ref{thm:main2} gives more general and direct condition for the question.
   We will give a proof of theorem \ref{thm:main2} in section 5, and in this section we also give concrete method to create hyperbolic knots via Nielsen-Thurston theory.
 
\begin{ack}
 The author would like to thank his adviser Toshitake Kohno for many helpful suggestions and advices for preparing this manuscript. He also would like to thank Eiko Kin for useful conversations.
\end{ack}
 
\section{Braid ordering and Dehornoy floor}

   In this section we review some of fundamental facts of Dehornoy ordering and Dehornoy floor.
Braid group $B_{n}$ is, by definition, a group which is presented by
\[ \langle \sigma_{1},\cdots,\sigma_{n-1}| \sigma_{i}\sigma_{j}\sigma_{i}=\sigma_{j}\sigma_{i}\sigma_{j} \textrm{ if } |i-j|=1, \; \sigma_{i}\sigma_{j}=\sigma_{j}\sigma_{i} \textrm{ if } |i-j|\geq 2  \rangle .\]
 Geometrically, elements of braid group $\beta \in B_{n}$ are represented by braided $n$-strands which move from above to below, oriented downward. By connecting their endpoints of strands canonically, we obtain an oriented link. We denote this link by $\widehat{\beta}$ and call {\it closed braid} or {\it closure} of $\beta$.

   We call the braid $(\sigma_{1}\sigma_{2}\cdots\sigma_{n-1})(\sigma_{1}\sigma_{2}\cdots\sigma_{n-2})\cdots(\sigma_{1}\sigma_{2})(\sigma_{1})$ {\it Garside's fundamental braid} and denote it by $\Delta$. The Garside fundamental braid$\Delta$ plays an important role in braid groups. 
 It is known that $\Delta^{2}$ generates the center of $B_{n}$, which is an infinite cyclic group.
 For these basic facts about braid group, see \cite{b}.
 
\subsection{Dehornoy ordering, Dehornoy floor}

   In this section we give a brief exposition of left-invariant total ordering of braid groups called Dehornoy ordering, and define Dehornoy floor. Dehornoy floor is first introduced in \cite{mn}, though they do not use term Dehornoy floor. This is a measure of how far braid is from the identity elements with respect to Dehornoy ordering and can be seen as some kinds of complexity of braids. 
We describe fundamental results about Dehornoy floor, proved in \cite{mn},\cite{ma}. 

   A braid $\beta \in B_{n}$ is called $\sigma$-$positive$ if $\beta$ can be represented by a word which contains at least one $\sigma_{i}$, and contains no $\sigma_{1}^{\pm1},\sigma_{2}^{\pm1}\cdots,\sigma_{i-1}^{\pm 1},\sigma_{i}^{-1}$ for some $1\leq i \leq n-1$. 
we say $\alpha < \beta$ is true if and only if the braid $\alpha^{-1}\beta$ is $\sigma$-positive.
It is known that the relation $<$ defines left-invariant total ordering of $B_{n}$; That is, $<$ is total ordering and if braids $\alpha,\beta \in B_{n}$ satisfy $\alpha < \beta $ , then for all braid $\gamma \in B_{n}$, $\gamma \alpha < \gamma \beta$ holds.  
We call this ordering {\it Dehornoy ordering}. 
   Dehornoy ordering has following nice property, called {\it property $S$} or {\it Subword property}.
   
\begin{prop}[Property S (\cite{ddrw})]
\label{prop:Property-S}
For any braids $\beta_{1},\beta_{2} \in B_{n}$ and $1\leq i \leq n-1$,
\[ \beta_{1}\sigma_{i}\beta_{2} > \beta_{1}\beta_{2} > \beta_{1}\sigma_{i}^{-1}\sigma_{2} \]
holds. 
\end{prop}

   As we mentioned at introduction, there are many other interpretation and definition of Dehornoy ordering. See \cite{ddrw} for those other definitions and meanings. We will describe one of the geometrical definition later.  
Using Dehornoy ordering, we define Dehornoy floor of braids $[\beta]_{D}$ as follows.

\begin{defn}
  Dehornoy floor $[\beta]_{D}$ of braid $\beta \in B_{n}$ is a minimal non-negative integer $m$ which satisfies
  $\beta \in (\Delta^{-2m-2},\Delta^{2m+2})$, where $(\Delta^{-2m-2},\Delta^{2m+2})= \{ \alpha \in B_{n} \;| \; \Delta^{-2m-2} < \alpha  <\Delta^{2m+2} \}$. 
\end{defn}

   Now we remark one important facts.
   We must be careful that when we think about Dehornoy floor, which $B_{n}$ a braid $\beta$ belongs to is very important.
For example, for a braid $\beta = (\sigma_{1}\sigma_{2})^{4}$, $[\beta]_{D}=1$ if we consider $\beta \in B_{3}$, and $[\beta]_{D}=0$ if we consider $\beta \in B_{4}$. Fortunately, if we use braid groups to describe links, the number of strands are always implicit, so there might be no confusion about the number of braid strands. So in most cases, we omit to write which $B_{n}$ a braid $\beta$ belongs to. 
   we also remark that Dehornoy floor is not a conjugacy invariant of braids. For example, 3-braid $\Delta^{2}\sigma_{1}\sigma_{2}^{-1}$ has Dehornoy floor $1$ but its conjugation $\Delta^{-1}(\Delta^{2}\sigma_{1}\sigma_{2}^{-1})\Delta = \Delta^{2}\sigma_{2}\sigma_{1}^{-1}$ has zero Dehornoy floor. This is mainly because Dehornoy ordering is not right-invariant; That is, it does not hold that for every $\gamma \in B_{n}$, $\alpha\gamma < \beta\gamma$ whenever $\alpha < \beta$. Thus, Dehornoy floor is far from knot invariant.  
   
   The most important property of Dehornor floor is that it can be seen as some kinds of restriction for admissibility of braid moves. We summarize the basic properties of Dehornoy floor, all of which are proved in \cite{mn}, or \cite{ma}.

\begin{prop}[\cite{mn},\cite{ma}]
\label{prop:dehornoyfloor}
   Let $\alpha,\beta \in B_{n}$. Then following holds.
\begin{enumerate}
\item If braid $\beta \in B_{n}$ is represented by braid word which contains $s$ occurrence of $\sigma_{1}$ and $k$ occurrence of $\sigma_{1}^{-1}$, then $[\beta]_{D} < max\{s,k\}$.
\item $|[\beta]_{D}-[\beta']_{D}|\leq 1$ if $\beta$ and $\beta'$ are conjugate.
\item $[\alpha \beta]_{D} \leq 1+[\alpha]_{D}+[\beta]_{D}$.
\item If closure of braid $\beta \in B_{n}$ admit destabilization or exchange move, $[\beta]_{D} \leq 1$.
\item For every $n$, there exist positive integer $r(n)$ such that for every $\beta \in B_{n}$, if $[\beta]_{D}>r(n)$ then $\widehat{\beta}$ is the unique closed $n$-braid representative of link types represented by its closure. That is, if $\widehat{\beta} = \widehat{\beta'}$ holds for some $\beta' \in B_{n}$, $\beta$ and $\beta'$ are conjugate.
\end{enumerate}
\end{prop}

   We remark that Dehornoy floor is not a conjugacy invariant but proposition \ref{prop:dehornoyfloor} (2) shows taking conjugate changes Dehornoy floor at most by one.

\begin{rem}
   Our definition of Dehornoy floor is slightly different from one in \cite{ma}. Our definition is rather after \cite{mn}. We adopt this definition because for application of knot theory, this formulation seems to be more natural and makes statement of results simple. 
\end{rem}

\begin{rem}
There exist an infinite family of left-invariant total ordering of braid groups, containing Dehornoy ordering, called {\it Thurston-type ordering} (See \cite{ddrw} or \cite{sw}). Most of results in this paper can be stated by using Thurston-type ordering instead of using Dehornoy ordering, especially using "Thurston floor" defined by the similar way.
\end{rem}

\subsection{The geometric meaning of Dehornoy ordering and Dehornoy floor}

  Now we give a geometric description of Dehornor ordering given in \cite{fgrrw}. Let $\Gamma$ be a graph in $D_{n}$ consisting of horizontal line which connects puncture points. See figure \ref{fig:braidorder_meaning}. 
We identify braid groups as a relative mapping class group $MCG(D_{n},\partial D_{n})$, which is a group of isotopy classes of homeomorphisms of $D_{n}$ whose restriction to boundary $\partial D_{n}$ is an identity. For braid $\alpha, \beta$, consider the image of graph $\Gamma$. By performing isotopy, we can assume $\alpha(\Gamma)$ and $\beta(\Gamma)$ have minimal intersections. 
If $\alpha \neq \beta$, these graphs must be different, so at some puncture point two graphs will diverge. We say $\beta >_{geom} \alpha$ if at divergence point, $\beta(\Gamma)$ moves left side of $\alpha(\Gamma)$. See figure \ref{fig:braidorder_meaning}.
In \cite{fgrrw} it is proved that this relation $>_{geom}$ coincide with Dehornoy ordering $>$.

Since $\Delta^{2}$ corresponds to rotation of the disc $D_{n}$, Dehornoy floor of $\beta$ is the number of how many times the initial segment of graph $\beta(\Gamma)$ winding around $\partial D_{n}$ under neglecting orientation.

\begin{figure}[htbp]
 \begin{center}
  \includegraphics[width=100mm]{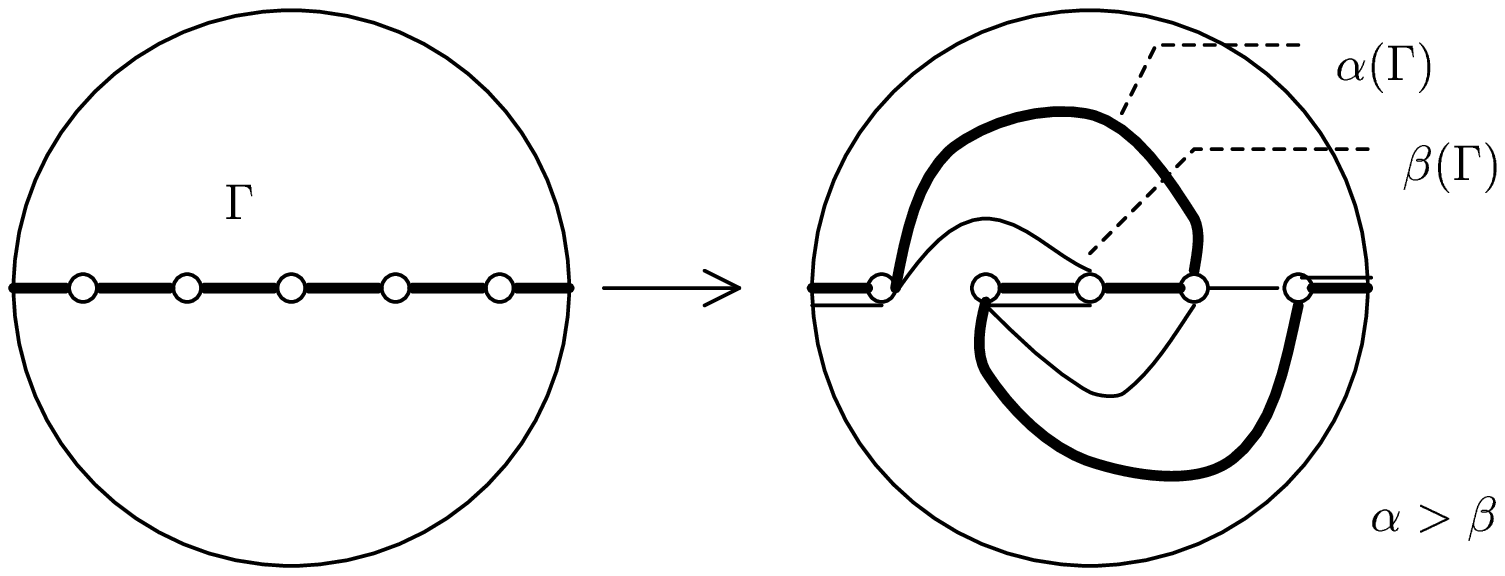}
 \end{center}
 \caption{Graph $\Gamma$ and examples}
 \label{fig:braidorder_meaning}
\end{figure}

\subsection{More precious estimation of Dehornoy floor}
   For later use, we prove slightly more precious estimation of Dehornoy floor than proposition \ref{prop:dehornoyfloor} (1).  
\begin{prop}
\label{prop:newestimation}
If a braid $\beta$ is conjugate to a braid $\alpha$ which is represented by braid word which contains $s$ occurrence of $\sigma_{1}$ and $k$ occurrence of $\sigma_{1}^{-1}$, then $[\beta]_{D} < max\{s,k\}$.
\end{prop}
\begin{proof}
   We show $\beta < \Delta^{2k}$. The proof of $\Delta^{-2s}< \beta$ is identical. By assumption, we can write $\beta = \gamma^{-1}(W_{0}\sigma_{1}W_{1}\sigma_{1}\cdots W_{k})\gamma$ where $W_{i}$ is a braid word which contains no $\sigma_{1}$ and $\gamma \in B_{n}$. We remark that we can assume $W_{0}$ is empty because we can amalgamate $W_{0}$ into conjugating element $\gamma$.
 
   Let $\Delta' = (\sigma_{2}\sigma_{3}\cdots \sigma_{n-1})^{n-1}$. 
From property S, we can find a braid $\beta' = \gamma^{-1}(\sigma_{1}\Delta'^{2p}_{(1)}\sigma_{1}\cdots \Delta'^{2p}_{(k)})\gamma > \beta$ by inserting $\sigma_{1},\cdots \sigma_{n-1}$ to braid $\beta$, for sufficiently large $p>0$. We show $\beta' <\Delta^{2k} $, that is equivalent to an inequality
$ (\sigma_{1}\Delta'^{2p}_{(1)}\sigma_{1}\cdots \Delta'^{2p}_{(k)})\gamma < \Delta^{2k} \gamma$. 
First we consider the case $k=1$. We use the geometric interpretation of Dehornoy ordering. Let $D'$ be a sub-disc of $D_{n}$ which contains $2,\cdots n$-th puncture of $D_{n}$.
Since $\Delta'^{2p}$ is p-times twists around $\partial D'$ and $\Delta^{2}$ is a rotation along $\partial D_{n}$, we conclude that the image of $\Delta^{2}\gamma(\Gamma)$ moves more left than $\sigma_{1}\Delta'^{2p}\gamma(\Gamma)$ for all $p$, thus $\sigma_{1}\Delta'^{2p}\gamma <\Delta^{2}\gamma$. Induction about $k$ gives desired result.
\end{proof}

   This proposition gives more accurate estimation of Dehornoy floor than using lemma \ref{prop:dehornoyfloor} (1).
\begin{rem}
   Proposition \ref{prop:newestimation} makes results in \cite{mn} a bit stronger form because in \cite{mn} they use estimation via lemma \ref{prop:dehornoyfloor} (1). Especially, we conclude that if $\widehat{\beta}$ admit destabilization or exchange move, then $[\beta_{D}]=0$.
\end{rem}

\section{Braid Foliation}

   In this section, we summarize basic machinery of Birman-Menasco's braid foliation theory in case of an embedded essential surface. For details of these techniques, see \cite{bf} or \cite{bm2}.

   Fix an unknot $A \in S^{3}$, called {\it axis} and choose a meridian fibration $H = \{ H_{\theta}| \theta \in [0,2\pi]\}$ of the solid torus $S^{3} \backslash A$.
A link $L$ in $S^{3} \backslash A$ is called {\it closed braid} with axis $A$ if $L$ intersects every fiber $H_{\theta}$ transversely and each fiber is oriented so that all intersections of $L$ are positive. It is easy to see closed braid $L$ intersects every fiber $H_{\theta}$ in the same number of points, and we call this number {\it braid index} of $L$. 
Note that closed braids $\widehat{\beta}$ obtained by a braid $\beta\in B_{n}$ as described in section 2 is indeed closed braid with $z-$axis and braid index is $n$. Conversely, if we cut solid torus $S^{3}\backslash A$ along the fiber $H_{0}$, we obtain a braid $\beta$. It is known that isotopy in $S^{3}\backslash A$ does not change conjugacy class of $\beta$ and if the isotopy fixes $H_{0}\cap \widehat{\beta}$, obtained braid $\beta$ does not change as a element of braid group though geometrical configuration of strands are varied.  

   A closed embedded surface $F$ in $S^{3} \backslash A$ is called {\it essential} if $F$ is incompressible and non-boundary parallel. 
Let $F$ be a essential closed surface in $S^{3} \backslash L$. Then the intersections of fiber $\{H_{\theta}\}$ with $F$ induce a singular foliation of $F$. Leaves of this foliation are connected components of intersection with fiber. Braid foliation techniques are, in short, changing this foliation simpler as possible and obtain standard and simple position of the surface and braids. 
By the argument in \cite{bf}, $F$ can be isotoped to "general" position with respect to the fibration satisfying:
\begin{description}
\item[(1)] Axis $A$ pierces $F$ transversely in finitely many points.
\item[(2)] For each point $v \in A\cap F$, there exists neighborhood $N_{v}$ of $v$ such that $F\cap N_{v}$ is radially foliated disc.
\item[(3)] All but finitely many fibers $H_{\theta}$ intersect $F$ transversely, and each of the exceptional fiber is tangent to $F$ exactly one point. Moreover, each point of tangency is saddle tangency and lies in the interior of $F \cap H_{\theta}$.
\end{description}

   Notice that the condition above is a bit stronger than usual general position arguments, since usual general position arguments only tell us that each tangency is local minimum, maximum, or saddle. In fact, this strong sense of general position is achieved from usual general position by deleting local minimum or maximum tangency.

   We say fiber $H_{\theta}$ {\it regular} if $H_{\theta}$ transverse $F$ and {\it singular}
if $H_{\theta}$ tangent to $F$. 
In the foliation of $F$, every non-singular leaf of induced foliation is either simple closed curve or arc which has both endpoints on axis $A$. We call former {\it c-circle} and the latter {\it b-arc}. Each b-arc in a fiber $H_{\theta}$ separates $H_{\theta}$ to two components, and we say b-arc is {\it essential} if both of these components are pierced at least once by $L$. For c-circle in a fiber $H_{\theta}$, c-circle bound a disc $D$ in $H_{\theta}$ and we say c-circle is {\it essential} if $D$ is pierced at least once by $L$.
As shown in \cite{bf}, we can assume further more condition about foliation of $F$.
\begin{description}
\item[(4)] every fiber is essential. 
\item[(5)] every c-circle is homotopically non-trivial.
\end{description}

   This condition is achieved by deleting all inessential leaves or homotopically trivial c-circle. From now on, we always assume that $F$ satisfies condition (1)$-$(5).

   We call an intersection point of $A$ with $F$ {\it vertex}. For each vertex $p$, the valance of vertex $p$ is, by definition, the number of singular leaves which pass $p$. 
We say singular point is $bb-singularity$ if the singular point derived from two b-arcs.
$bc-singularity$ and $cc-singularity$ are defined by the same way. Each type of singularity has a neighborhood in $F$ as shown in figure \ref{fig:regions}, and we say such a neighborhood {\it regions}. 

   A foliation of $F$ is classified into following three types according to its containing leaves; We say $F$ is {\it tiled} if $F$ is foliated entirely b-arc and {\it circular} if $F$ is foliated entirely c-circle. The remaining case, if $F$ is foliated by both b-arcs and c-circles, we say $F$ is {\it mixed}.
In tiled surface, there are only bb-singularity and circular surface has only cc-singularity. Mixed foliation might have all types of singularity and always contains bc-singularity.

\begin{figure}[htbp]
 \begin{center}
  \includegraphics[width=120mm]{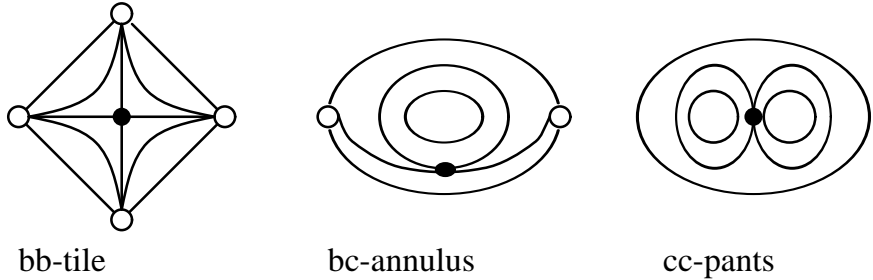}
 \end{center}
 \caption{bb-tile,bc-annulus,cc-pants}
 \label{fig:regions}
\end{figure}
 
\section{Proof of theorem \ref{thm:main1}}

   Now we are ready to prove theorem \ref{thm:main1}. From now on, we always assume $F$ satisfies all conditions (1)-(5) described in section 3. First we consider the case $g=0$. In this case, $F$ is never circular foliated. The main result of \cite{bm1} shows in this case $\widehat{\beta}$ admit exchange move if $F$ is tiled and $\beta$ has word representatives which contains no $\sigma_{k}^{\pm1}$ for some $k$, so $[\beta]_{D}=0$. Thus, in this case theorem holds. Thus, we always assume $g>0$.
 
\subsection {Tiled essential surface}

   First we prove statement (1), tiled case. Assume $F$ is tiled.
Since $F$ is tiled, every singular point is bb-singularity, and $F$ can be decomposed into bb-tiles. A decomposition into bb-tiles determines cellular-decomposition of $F$, and we call this {\it tiling}. 
The notion of vertex and valance we defined in section 3 coincides to usual definition of vertex and valance in this cellular decomposition. 
 
   First we estimate a valance of each vertex, using Euler characteristic argument which is standard in braid foliation theory.

\begin{lem}[Euler characteristic formula, Tiled case]
\label{lem:eulerform-tiled}
  Let $\beta$ be a braid and $F$ be a tiled essential genus $g$ closed surface in the complement of closed braid $\widehat{\beta}$, and $V(i)$ be a number of vertex whose valance is $i$. then,
\[ \sum_{v=1}^{3}(4-v)V(v)+8g-8= \sum_{v=4}^{\infty}(v-4)V(v)\]
 holds.
 \end{lem}
\begin{proof}
   Let $E,R$ be a number of 1-cell and 2-cell of tiling.
Since $F$ has an Euler characteristic $-2g+2$, $(\sum_{i=1}^{\infty}V_{i}) - E + R = -2g+2$.
On the other hand, each 1-cell is the boundary of exactly two 2-cells and each 1-cell has two vertices as their boundary ,so $2R=E$ and $(\sum_{i=1}^{\infty}iV_{i}) = 2E$. 
By combining these equation, we establish the claim.

\end{proof}
   Since each term of equation in lemma \ref{lem:eulerform-tiled} is positive, lemma \ref{lem:eulerform-tiled} implies $V(i)$ are bounded except $V(4)$. Next we see that Dehornor floor can be estimated from above by a valance of vertex.

\begin{lem}
\label{lem:estimate-tiled}
   Let $F$ be a tiled essential genus $g$ closed surface in the complement of closed braid $L=\widehat{\beta}$. If $F$ contains a valance $v$ vertex, then $[\beta]_{D} < \frac{v}{2}+1$.
\end{lem}
\begin{proof}

   Let $p$ be a valance $v$ vertex and $ \{H_{\theta_{1}},H_{\theta_{2}},\cdots ,H_{\theta_{v}} |\; \theta_{i}<\theta_{i+1} \} $ be a sequence of singular fibers such that each $H_{\theta_{i}}$ contains singular leaves which pass the vertex $p$. Let $b_{\theta}$ be a b-arc in the fiber $H_{\theta}$ which passes $v$. By taking isotopy near each singular fiber $H_{\theta_{i}}$, we can always assume 
that for sufficiently small $\varepsilon > 0 $, braiding of strands does not occur in the interval $[\theta_{i}-\varepsilon, \theta_{i}+\varepsilon]$ and each fiber $H_{\theta_{i} \pm \varepsilon}$ is pierced by $L$ exactly the same points $\{p_{1},p_{2},\cdots,p_{n}\}$. Notice that this modification can be done without changing $H_{0}\cap \widehat{\beta}$ so this modification does not change braid $\beta$ as a element of braid groups. After this modification, we can always consider sub-braiding in each interval $[\theta_{i} - \varepsilon,\theta_{i}+\varepsilon]$, $[\theta_{i}+\varepsilon, \theta_{i+1}-\varepsilon]$.
Since $\varepsilon$ is chosen to sufficiently small, we can assume there exist no other singular points in each interval $[\theta_{i}-\varepsilon, \theta_{i}+\varepsilon]$. 

Now we see a sequence of $b_{\theta}$ as a move of b-arc in $D^{2}$. Then we can also assume by taking isotopy of surface and braids, in the interval $[\theta_{i}+\varepsilon, \theta_{i+1}-\varepsilon]$, $b_{\theta}$ do not move when we see $b_{\theta}$ as a move of the arcs. 
   
 First we investigate a sub-braiding $\beta_{1}$ in $[\theta_{1}+\varepsilon, \theta_{3}-\varepsilon] $.
 By taking appropriate conjugacy of $\beta_{1}$, we can assume $F \cap H_{\theta_{1} + \varepsilon}$ consists of several number of vertical arcs as depicted in left of figure \ref{fig:model} and bb-singularity in $\theta_{2}$ is modeled as shown in the right of figure \ref{fig:model}. We denote this braid by $\beta'_{1}$

  Then, by seeing braids as a move of puncture point of disc, we can write this modified braiding $\beta_{1}'$ in $[\theta_{1}+\varepsilon, \theta_{3}-\varepsilon]$ as shown in figure \ref{fig:braiding}, 
which contains only one $\sigma_{1},\sigma_{1}^{-1}$. The original braiding $\beta_{1}$ in $[\theta_{1}+\varepsilon, \theta_{3}-\varepsilon]$ is conjugate to the braid $\beta_{1}'$ so proposition \ref{prop:newestimation} shows $[\beta_{1}]_{D} =0$.

\begin{figure}[htbp]
 \begin{center}
  \includegraphics[width=110mm]{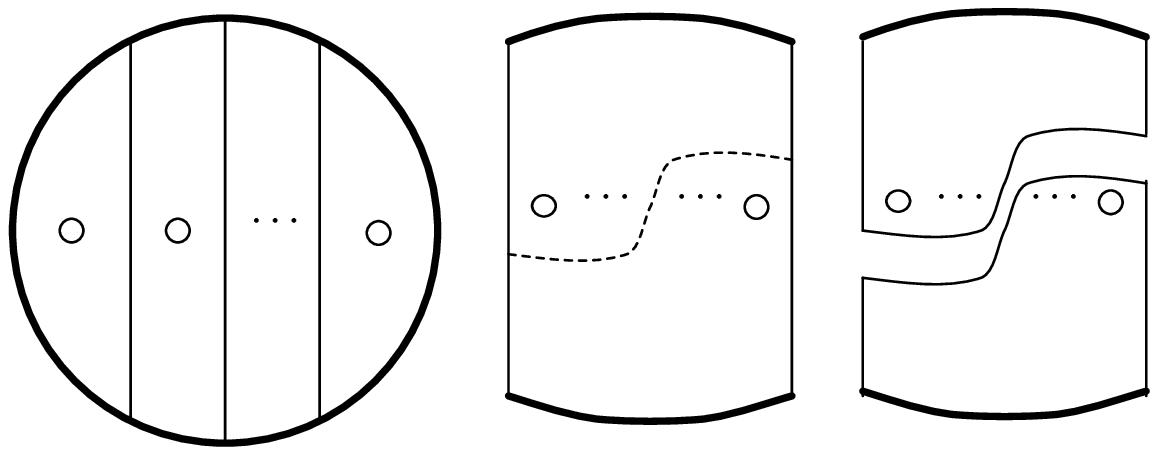}
 \end{center}
 \caption{Model of intersection and bb-singularity}
 \label{fig:model}
\end{figure}
\begin{figure}[htbp]
 \begin{center}
  \includegraphics[width=80mm]{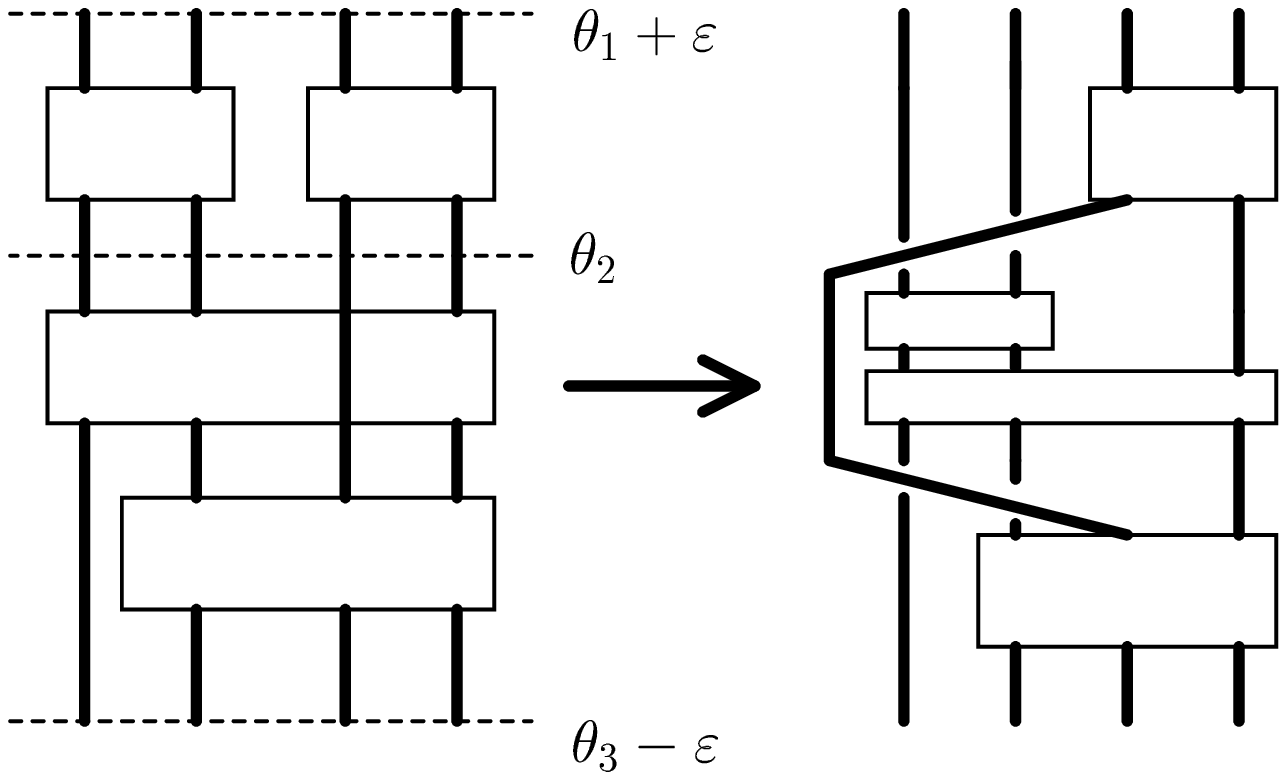}
 \end{center}
 \caption{Braiding in $[\theta_{1}+\varepsilon,\theta_{3}-\varepsilon]$: bb-singularity case}
 \label{fig:braiding}
\end{figure}

   Iterating this argument, we conclude that sub-braid $\beta_{i}$ in the interval $[\theta_{2i-1}+\varepsilon, \theta_{2i+1}-\varepsilon]$ for $i=1,\cdots ,\frac{v}{2}$ has zero Dehornoy floor. Since original braid $\beta$ is conjugate to the braid $\beta_{1}\beta_{2}\cdots\beta_{v \slash 2}$, by proposition \ref{prop:dehornoyfloor}, $ [\beta]_{D}< \frac{v}{2} $.
\end{proof}

   Combining above two lemmas, now it is easy to estimate Dehornor floor via information of tiling.

\begin{proof}[Proof of Theorem \ref{thm:main1}(1)]
   Assume $F$ is tiled. Let $V(i)$ be the number of valance $i$ vertices. If some of $V(i)$ is non-zero for $i \leq 3$, then lemma \ref{lem:estimate-tiled} means $[\beta]_{D}< 2 < g+1$, so conclusion holds. Therefore we can assume $V(i)=0$ for $i \leq 3$, and Euler characteristic formula in lemma \ref{lem:eulerform-tiled} become 
\[ 8g-8= \sum_{v=4}^{\infty}(v-4)V(v)\]

   Since $g$ is not zero, there exist at least two b-arcs in each regular fiber, so at least four vertices in the tiling of $F$. Therefore the minimal valance of vertices is smaller than $max\{ 2g+2,4 \}$, that is, $2g+2$. Then lemma \ref{lem:estimate-tiled} establishes $[\beta]_{D}< g+2$. 
\end{proof}

\subsection{Mixed foliated surface}
   Next we proceed to $F$ is mixed foliated case. The strategy is almost the same as tiled case, but we need to do some additional arguments to evaluate valance of vertices because the decomposition by regions do not determine cellular decomposition. First we see the same conclusion of lemma \ref{lem:estimate-tiled} holds.
   
\begin{lem}
\label{lem:estimate-mixed}
  Let $F$ be a mixed foliated essential genus $g(\neq0)$ closed surface in the complement of closed braid $\widehat{\beta}$. If $\beta$ contains a valance $v$ vertex, then $[\beta]_{D} < \frac{v}{2}+1$.
\end{lem}
\begin{proof}
   Let $p$ be a valance $v$ vertex. We can use the same arguments appeared in proof of lemma \ref{lem:estimate-tiled}. The main difference is a presence of c-circles and bc-singularity around $p$. 
   First we modify closed braid as in the proof of lemma \ref{lem:estimate-tiled} so that we can consider sub-braiding in each interval  $[\theta_{i}-\varepsilon, \theta_{i}+\varepsilon]$,  $[\theta_{i}+\varepsilon, \theta_{i+1}-\varepsilon]$.
   Assume that singular point in the fiber $H_{\theta_{2}}$ is bc-singularity. Then by taking appropriate conjugation of sub-braiding $[\theta_{1}+\varepsilon, \theta_{3}-\varepsilon]$, bc-singularity is now modeled as in the figure \ref{fig:model2}. Then braiding in the interval $[\theta_{1}+\varepsilon, \theta_{3}-\varepsilon]$ can be written as shown in figure \ref{fig:braiding2}, which contains only one $\sigma_{1}^{\pm 1}$. Thus, the same argument of the proof of lemma \ref{lem:estimate-tiled} gives desired estimation.
   
\begin{figure}[htbp]
 \begin{center}
  \includegraphics[width=80mm]{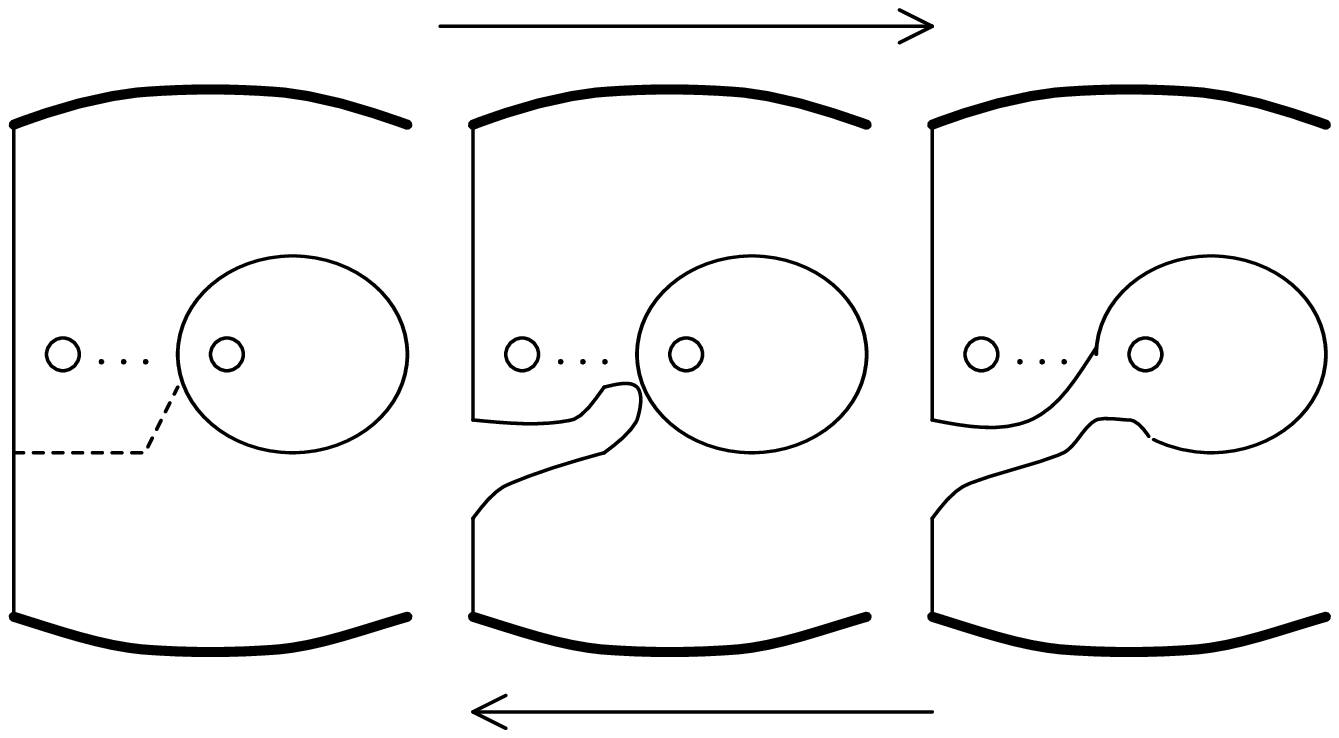}
 \end{center}
 \caption{Model of bc-singularity}
 \label{fig:model2}
\end{figure} 

\begin{figure}[htbp]
 \begin{center}
  \includegraphics[width=80mm]{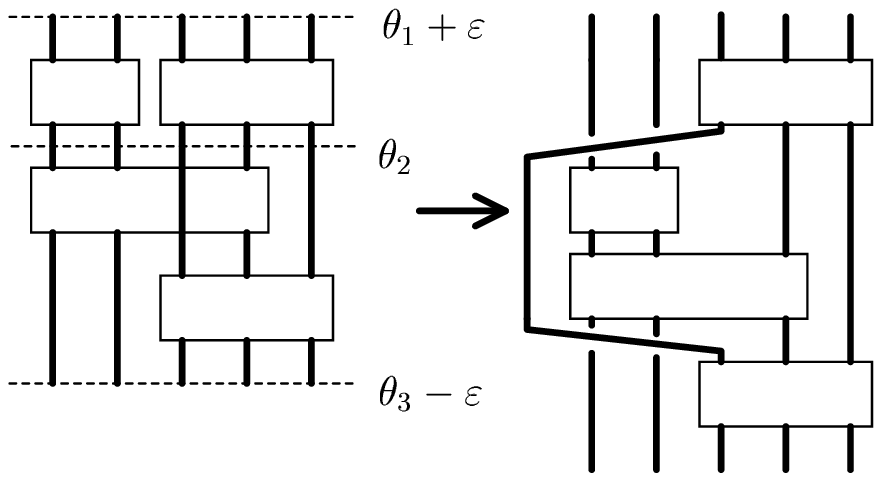}
 \end{center}
 \caption{Braiding in $[\theta_{1}+\varepsilon, \theta_{3}-\varepsilon]$: bc-singularity case}
 \label{fig:braiding2}
\end{figure}

\end{proof}

   Then, we want to estimate the minimal valance of vertices. 
Since there might be all types of singularity in the foliation of $F$, in general decomposition into regions dose not define usual sense of cellular decomposition: bc-region is annulus, and cc-region is pair of pants. However, in special case of normal mixed foliation, defined below, one can constructs canonical cellular decomposition via decomposition into regions as shown in \cite{bm2}. First we describe this. 

   We say mixed foliation is {\it normal} if there exists no cc-singularity.
In normal mixed foliation case, we can construct canonical cellular decomposition as follows. 
First we see that in normal mixed foliation, bc-annuli always occur in pairs. This is because the c-circle boundary of bc-annuli must be attached to c-circle boundary of another bc-annuli, since there is no cc-regions and boundary of bb-tile are all b-arcs. Now, take annulus $W$ obtained by attaching two bc-annuli along their c-circle boundaries. Each component of $\partial W $ has two vertices. By cutting $W$ along the arc which connects two vertices which lies different components of $\partial W$, as shown in figure \ref{fig:be-tile}, we obtain two 2-cells, called {\it be-tile}. These cutting arcs are called {\it e-edge}.
 Now decomposition into be-tile and bb-tile determines cellular decomposition of $F$.
 
\begin{figure}[htbp]
 \begin{center}
  \includegraphics[width=120mm]{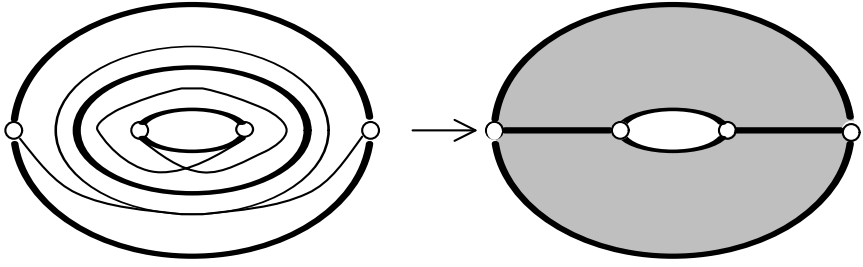}
 \end{center}
 \caption{Construction of be-tile}
 \label{fig:be-tile}
\end{figure}

   Note that there might exist be-tile which has no singular point or two singular points, but from construction of be-tile, be-tile with zero singularity and be-tile with two singularity must occur in pairs in each neighbor of vertex, so the notion of vertex and valance defined in section 3 coincide with the usual meaning of vertex and valance in this cellular decomposition.
   Then we see the same Euler characteristic formula holds: 
 
\begin{lem}
\label{lem:eulerform-nmixed}  
   Let $F$ be a normal mixed foliated surface with genus $g \neq 0$ and $V(v)$ be a number of vertex whose valance is $v$ in cellular decomposition described above. Then,
 
\[\sum_{v=1}^{3}(4-v)V(v)+8g-8 = \sum_{v=4}^{\infty}(v-4)V(v)\]
holds.
\end{lem}
\begin{proof}
    A be-tile has exactly four 1-cells in its boundary and each 1-cell is boundary of two 2-cells, so the same argument in proof of lemma \ref{lem:eulerform-tiled} gives desired result.
\end{proof}

   So we can estimate minimal valance of vertices in case of normal mixed foliation by using lemma \ref{lem:eulerform-nmixed} as in the tiled case. To estimate general mixed foliated case, we perform "preimage surgery" to delete cc-singularities.
   Let $F$ be a mixed foliated surface and we denote $\mathcal{F} $ by the preimage of $F$; that is, $\mathcal{F}$ is a foliated surface forgetting information of embedding.
   
   Let $c$ be an one of the c-circle boundary of cc-pants $P$. $c$ is also a boundary of some other cc-pants or bc-annulus $Q$. We modify foliations of $P$ and $Q$ as shown in figure \ref{fig:preimagesurgery} by filling $c$ by discs and changes foliation in $P$ and $Q$, according to type of regions. Since modified regions are trivially foliated, these regions are amalgamated to some other regions and they disappears. Moreover, this operation can be done so that obtained surface is connected by choosing c-circle boundary $c$ appropriately and reduces genus of $\mathcal{F}$ because all c-circle in the foliation is non-null homotopic in $\mathcal{F}$. 
Thus, this modification deletes two regions $P$ and $Q$, hence two singular points. We call this operation {\it preimage surgery} because we can regard this modification as surgery of $\mathcal{F}$ along c-circle $c$. We remark that preimage surgery is not realizable as a surgery of embedded surface $F$, because we cannot choose take a disc bounded by $c$ without intersecting original surface $F$ in $S^{3}$.

   Since each preimage surgery reduces a genus of $\mathcal{F}$, after $k \leq g$ preimage surgery we obtain normal mixed foliated surface $\mathcal{F}'$. As we remarked before, $\mathcal{F}'$ cannot be realized as an embedded surface, but we can use the arguments of proof of lemma \ref{lem:eulerform-nmixed}, because the proof uses only information from cellular decomposition of $\mathcal{F}'$. Therefore we can use lemma \ref{lem:eulerform-nmixed} for $\mathcal{F}'$ to estimate valance of vertices of preimage-surgered surface $\mathcal{F}'$.
\begin{figure}[htbp]
 \begin{center}
  \includegraphics[width=120mm]{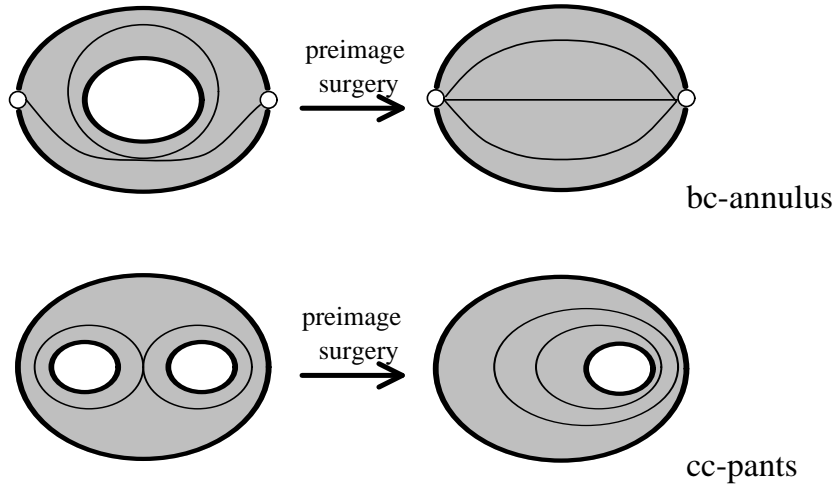}
 \end{center}
 \caption{Preimage surgery}
 \label{fig:preimagesurgery}
\end{figure}

\begin{proof}[Proof of Theorem \ref{thm:main1}(2)(3)]
    Assume $F$ is mixed foliated and after performing preimage surgery $k$-times described above, we create abstract foliated surface $\mathcal{F}'$ with no cc-singularity, that is, normal mixed foliated. 
Let $m$ be a minimal valance of vertices of $\mathcal{F}'$. Then, we trace preimage surgery procedure backwards. Each surgery deletes at most one bc-singularity, so original surface $F$ has at most more than $k$ bc-singular points than that of $\mathcal{F}'$. Since singular leaf containing bc-singularity has two vertices as its endpoints and there exist at least two vertices in $F$, we conclude minimal valance of vertices in $F$ is smaller than $m+k$. Thus, we only need to estimate $m$ to estimate minimal valance of vertices in original surface $F$.

   Let $V(v)$ be the number of valance $v$ vertices of $\mathcal{F}'$. If $V(i)$ is non-zero for some $1 \leq i \leq 3$, then form the argument above, we conclude the minimal valance of vertices of $F$ is smaller than $3+g$, so $[\beta]_{D} < \frac{3+g}{2} \leq 2g+1$, so we can assume $V(i)$ is zero for all $1 \leq i \leq 3$.
  In this case, the Euler characteristic formula for $\mathcal{F}'$ become 
  \[8(g-k)-8 > \sum_{i=4}^{\infty}(v-4)V(v).\]
     Mixed foliated means there exist at least one b-arc in the foliation, so there exist at least two vertices in $\mathcal{F}'$. Therefore the minimal value of valance of vertices in $\mathcal{F}'$ is less than $4(g-k)$. Therefore, from our earlier argument, minimal valance of vertices in $F$ is smaller than $4(g-k)+k = 4g-3k+1$. Now, lemma \ref{lem:estimate-mixed} establishes an inequality $[\beta]_{D}< 2g-\frac{3k}{2}< 2g+1$.
     
 From statement (1) and (2), we conclude if $[\beta]_{D} \geq 2g+1$, then $F$ is neither tiled nor mixed-foliated, hence circular-foliated. This completes proof of theorem \ref{thm:main1}
\end{proof}

\section{Nielsen-Thurston classification and the geometry of closed braid complements}

   In this section, we describe an application of theorem \ref{thm:main1} for genus one closed surface, namely, essential torus. A description of essential torus in closed braid complements is studied in \cite{bm2},\cite{n} by using braid foliation machinery. 
Since an essential torus deeply involves the geometry of link complements, we can extract some of geometrical information of closed braids from theorem \ref{thm:main1}.
\subsection{Nielsen-Thurston classification}
   First of all, we recall the Nielsen-Thurston classification of braids. 
Nielsen-Thurston theorem states that an isotopy class of orientation preserving homeomorphism $[f]$ of compact oriented surface $F$, that is, an element of mapping class group $MCG(F)$ of $F$, is classified into following three types:
\begin{enumerate}
\item {\it periodic}: there exists a positive integer $m$ such that $f^{m}$ is isotopic to the identity.
\item {\it reducible}: there exist an essential 1-submanifold $C$ of $F$, which is invariant under $f$.
\item {\it pseudo-Anosov}: $f$ is isotopic to a pseudo-Anosov homeomorphism.
\end{enumerate}

   We do not describe precise definition of pseudo-Anosov homeomorphism, because we do not use it.(see \cite{flp} for more details about Nielsen-Thurston theory.)
   Let $D_{n}$ be a $n$-punctured disc. Since braid group $B_{n}$ is identified with the group of boundary-fixing isotopy class of homeomorphisms of $D_{n}$ which is identity on $\partial D_{n}$, there exist natural projection $\pi:B_{n} \rightarrow MCG(D_{n})$ regarding an element of $B_{n}$ as an isotopy class of homeomorphism of $D_{n}$ whose restriction on $\partial D_{n}$ are not necessarily identity.
 Therefore we can define the Nielsen-Thurston classification of braids by taking pull-back of the Nielsen-Thurston classification of $MCG(D_{n})$.
Concretely, the Nielsen-Thurston classification of $B_{n}$ can be described as following way:
\begin{enumerate}
\item {\it periodic} : some powers of $\beta$ is equal to some powers of Garside fundamental braid $\Delta$. 
\item {\it pseudo-Anosov} : its representing homeomorphism is pseudo-Anosov.
\item {\it reducible}: its representing homeomorphism is reducible.
\end{enumerate}

   Periodic braids are the most simple braids and it is known that periodic braids are conjugate to some powers of braid $(\sigma_{1}\sigma_{2}\cdots\sigma_{n-1})$ or $(\sigma_{1}\sigma_{2}\cdots\sigma_{n-1}\sigma_{1})$, which represents rotation of disc.

   We remark that the Nielsen-Thurston classification of mapping class groups has very beautiful connection to the geometry of its mapping torus. Let $F$ be a closed surface, $\varphi$ be a homeomorphism of $F$ and $F_{\varphi}$ be the mapping torus of $\varphi$. The topology of mapping torus $F_{\varphi}$ only depends on the isotopy class of $\varphi$.
   
   It is known that $F$ admit a complete hyperbolic structure if and only if $\varphi$ is pseudo-Anosov and $F_{\varphi}$ is Seifert-fibered if and only if $\varphi$ is periodic. Moreover, $F_{\varphi}$ contains an essential torus (hence decomposable with respect to geometric decomposition) if and only if $\varphi$ is reducible. Thus, there exist one-to-one correspondence between the Nielsen-Thurston classification of mapping class group and the geometry of its mapping torus.
 Hence, in this regard, theorem \ref{thm:main2} can be seen as some generalization of this well-known fact in the knot theory settings. 

\subsection{Proof of Theorem \ref{thm:main2}}

   First we investigate non-braided satellite knots. Recall that a satellite knot $K$ is called braided-satellite if its complements has essential tori $T$ which is circular foliated under appropriate choice of closed braid representatives of $K$.
In general it is difficult to handle non-braided satellite knot in braid theory, but fortunately they do not appear under assumption of theorem \ref{thm:main2}.

\begin{prop}
\label{prop:nbdsatellite}
Let $\beta \in B_{n}$ be a braid such that $K=\widehat{\beta}$ is a non-braided satellite knot. Then, $[\beta ]_{D} < 3 $. 
\end{prop} 
    
\begin{proof}
   Assume $[\beta]_{D} \geq 3$. Then theorem \ref{thm:main2} shows all of essential tori in the complement of $K$ is circular, contradicts to assumption $K$ is non-braided satellite.
\end{proof}

\begin{rem}
In \cite{mn}, it is proved that if a closure of $\beta$ is a composite knot, then $[\beta]_{D}<2$. Proposition \ref{prop:nbdsatellite} can be seen some kind of generalization of this result because composite knots are special types of non-braided satellite knots.
\end{rem}

 To prove theorem \ref{thm:main2}, We need the following result of Menasco concerning about Torus knots.
 
\begin{thm}[Menasco(\cite{me1},\cite{me2})]
   Torus knots are exchange-reducible. That is, For every closed braid representative $\widehat{\beta}$ of $(p,q)$-torus knot,
there exists a sequence of closed braids 
\[ \widehat{\beta}= \widehat{\beta_{0}}\rightarrow \widehat{\beta_{1}} \rightarrow \cdots \rightarrow \widehat{\beta_{m}} \] 
 such that $\widehat{{\beta}_{i+1}}$ is obtained by performing exchange move or destabilization or isotopy of $\widehat{\beta_{i}}$, and $\beta_{m} = (\sigma_{1}\sigma_{2}\cdots\sigma_{p-1})^q$. 
\end{thm}

  We remark that this Menasco's theorem is also proved by braid foliation techniques. 

\begin{proof}[Proof of theorem \ref{thm:main2}]

   First assume $\beta \in B_{n}$ and $K =\widehat{\beta}$ is a torus knot. If $n$ is not the braid index of knot $K$, that is, if $K$ can be represented as a closure of other braids which have strictly smaller number of strands, then Menasco's theorem says $K$ admit exchange move or destabilization, hence by the proposition \ref{prop:dehornoyfloor} (5), we conclude $[\beta]_{D}\leq 1$, which contradicts our assumption $[\beta]_{D} \geq 3 $.
 Therefore we can assume that $n$ is braid index of $K$. In this case, also Menasco's theorem says $\beta$ is conjugate to braid $(\sigma_{1}\sigma_{2}\cdots \sigma_{n-1})^{p}$ for some $p$, hence $\beta$ is periodic.

   On the other hand, it is known that periodic braids are conjugate to braids $(\sigma_{1}\sigma_{2}\cdots\sigma_{n-1})^{s}$ or $(\sigma_{1}\sigma_{2}\cdots\sigma_{n-1}\sigma_{1})^{s}$ for some $s$ , both of which represent rotation of disc. Since closure of $(\sigma_{1}\sigma_{2}\cdots\sigma_{n-1}\sigma_{1})^{s}$ never become a knot, it is obvious that closure of periodic braids represent torus knots.
 
   Next assume $K=\widehat{\beta}$ is a satellite knot. 
   Let $T$ be a essential torus in the complement of $K$. Of course there might be some of choice for such torus, we can take arbitrary one. Our assumption $[\beta]_{D} \geq 3$ and theorem \ref{thm:main1} shows $T$ is circular-foliated, so intersection of $T$ with some fiber $H_{\theta}$ gives reducing 1-submanifolds of $\beta$ as a homeomorphism of $D_{n}$, hence $\beta$ is reducible. The converse is trivial.

   Since both classifications are exclusive, we conclude that $\widehat{\beta}$ is a hyperbolic knot if and only if $\beta$ is pseudo-Anosov.

\end{proof}

 As a corollary, we get an infinite, almost disjoint family of hyperbolic knots for each pseudo-Anosov element of mapping class group of punctured disc. Although it is known that "almost all" prime knots are hyperbolic knot, in general it is not so easy to determine given not is hyperbolic or not. Especially, most of known family of knots known to be hyperbolic are defined through knot diagrams such as alternativeness. We now construct many families of hyperbolic knot through braid-theoretical approach.

\begin{cor}
\label{cor:c1}
 Suppose $[f] \in MCG(D_{n})$ be a pseudo-Anosov element of mapping class group of $n$-punctured disc and $\pi: B_{n} \rightarrow MCG(D_{n})$ be the natural projection. Let $ P([f])=\{ \widehat{\beta} \;| \;\beta \in \pi^{-1} ([f]),[\beta]_{D} \geq 3 \} $. 

Then $P([f])$ consists of infinite number of distinct hyperbolic knots.
Moreover, let $[g] \in MCG(D_{n})$ be another pseudo-Anosov element which is not conjugate to $[f]$ in $MCG(D_{n})$.
Then, the intersection of $P([f])$ and $P([g])$ are finite. 

\end{cor}
\begin{proof}[Proof of corollary \ref{cor:c1}]
    Since $P([f])$ consists of pseudo-Anosov braids with Dehornoy floor larger than three, Theorem \ref{thm:main2} shows their closures are all hyperbolic knots. 
Proposition \ref{prop:dehornoyfloor} $(6)$ says braids having sufficiently large Dehornoy floor is a unique closed representative of its closure in braid index $n$, therefore $P([f])$ consists of infinite number of distinct hyperbolic knots and $P([f])$ and $P([g])$ have finite intersection unless $[g]$ is conjugate to $[f]$.
\end{proof}

\begin{rem}
   Since every hyperbolic knots are represented as a closure of pseudo-Anosov braids, one might expects our construction of hyperbolic knots produces {\it all} hyperbolic knots; That is, for every hyperbolic knots, there exists pseudo-Anosov mapping class $[f]$ such that $K \in P([f])$. Unfortunately, this too optimistic conjecture is not true.
   For example, we can prove figure-eight knot, the most simple hyperbolic knots, do not appear as an element of $P([f])$. More generally, we can show genus one hyperbolic knots cannot be constructed from our method. This is a corollary of results of our subsequent paper \cite{i}.
\end{rem}

 Finally, we state a special case of theorem \ref{thm:main2} which simplifies the situation.

\begin{cor}
\label{cor:c2}
Let $p$ be a prime and $\beta \in B_{p}$ be a braid whose closure is a knot and $[\beta]_{D}\geq 3$.
Then $\widehat{\beta}$ is a hyperbolic knot if and only if $\beta$ is non-periodic.
\end{cor}

\begin{proof}[Proof of corollary \ref{cor:c2}]
   Since $p$ is prime and $\widehat{\beta}$ is a knot, $\beta$ is never reducible braid.
   Assume $\widehat{\beta}$ is satellite. Then proposition \ref{prop:nbdsatellite}, $\widehat{\beta}$ is braided-satellite, so theorem \ref{thm:main2} means $\beta$ is reducible, which contradicts above. Thus, under the assumption of corollary, $\widehat{\beta}$ is either hyperbolic or torus knot and $\beta$ is either periodic or pseudo-Anosov, so theorem \ref{thm:main2} establishes the result. 
\end{proof}

   Although there is an algorithm to determine Nielsen-Thurston classification, it requires exponential times with respect to word length of braid. On the other hand, since to determine given braids are periodic or not is easy and it requires only quadratic times, corollary \ref{cor:c2} produces very many hyperbolic knots explicitly and rapidly. In particular, we conclude that for prime $p$, closure of braids $\beta \in B_{p}$ satisfying $[\beta]_{D} \geq 2$ are almost all hyperbolic knots.

\end{document}